\documentclass[letterpaper, 12pt]{article}
\usepackage[letterpaper, margin=1 in]{geometry}
\pdfoutput=1
\usepackage{amsmath, amsthm, amssymb}
\usepackage{graphicx}
\usepackage{epstopdf}
\usepackage[pdftex=true, bookmarks=true, bookmarksnumbered=true, backref=false, colorlinks=false, linkbordercolor={1 .75 .75},citebordercolor={.75 1 .75}, linktocpage=false]{hyperref}
\hypersetup{pdfauthor={Alex Bloemendal and B\'alint Vir\'ag},pdftitle={Limits of spiked random matrices I}}

\usepackage[longnamesfirst]{natbib}
\bibpunct[ ]{(}{)}{,}{a}{}{,}

\usepackage{setspace}
\usepackage{tocloft}

\usepackage{enumerate}

\allowdisplaybreaks[3]

\newcommand{\e}{\varepsilon}
\newcommand{\ph}{\varphi}

\newcommand{\del}{\partial}

\newcommand{\Z}{\mathbb Z}

\newcommand{\R}{\mathbb R}
\newcommand{\C}{\mathbb C}

\newcommand{\F}{\mathcal F}
\newcommand{\Hc}{\mathcal H}

\newcommand{\Pc}{\mathcal P}

\newcommand{\half}{\tfrac{1}{2}}
\newcommand{\abs}[1]{\left\lvert#1\right\rvert}

\newcommand{\norm}[1]{\left\lVert#1\right\rVert}
\newcommand{\Norm}[1]{\bigl\lVert#1\bigr\rVert}
\newcommand{\ip}[1]{\left\langle#1\right\rangle}
\newcommand{\Ip}[1]{\bigl\langle#1\bigr\rangle}

\DeclareMathOperator{\diag}{diag}

\renewcommand{\P}{\operatorname{\mathbf{P}}}
\newcommand{\1}{\mathbf{1}}
\DeclareMathOperator{\E}{\mathbf E}

\newcommand{\Hyw}[1]{\Hc_{y,w}\!\left(#1\right)}
\newcommand{\wt}{\widetilde}
\newcommand{\ol}{\overline}

\DeclareMathOperator{\Ai}{Ai}

\newcommand{\trb}{\tfrac{2}{\sqrt\beta}}

\newcommand{\weight}{{\textstyle\sqrt{1+\ol\eta}}}

\theoremstyle{plain}
\newtheorem{theorem}{Theorem}[section]
\newtheorem{lemma}[theorem]{Lemma}
\newtheorem{corollary}[theorem]{Corollary}
\newtheorem{proposition}[theorem]{Proposition}

\newtheorem{fact}[theorem]{Fact}
\theoremstyle{definition}
\newtheorem{definition}[theorem]{Definition}
\theoremstyle{remark}
\newtheorem{remark}[theorem]{Remark}

\numberwithin{equation}{section}
\defcitealias{BBP}{BBP}
\defcitealias{RRV}{RRV}

\title{Limits of spiked random matrices I}
\author{Alex Bloemendal\and B\'alint Vir\'ag}
\date{September 16, 2011}
\begin{document}

\maketitle

\begin{abstract} Given a large, high-dimensional sample from a spiked population, the top sample covariance eigenvalue is known to exhibit a phase transition.  We show that the largest eigenvalues have asymptotic distributions near the phase transition in the rank one spiked real Wishart setting and its general $\beta$ analogue, proving a conjecture of Baik, Ben Arous and P\'ech\'e (2005).  We also treat shifted mean Gaussian orthogonal and $\beta$ ensembles. Such results are entirely new in the real case; in the complex case we strengthen existing results by providing optimal scaling assumptions.  One obtains the known limiting random Schr\"odinger operator on the half-line, but the boundary condition now depends on the perturbation.  We derive several characterizations of the limit laws in which $\beta$ appears as a parameter, including a simple linear boundary value problem.  This PDE description recovers known explicit formulas at $\beta=2,4$, yielding in particular a new and simple proof of the Painlev\'e representations for these Tracy-Widom distributions.\end{abstract}

\bigskip\bigskip\bigskip
\setcounter{tocdepth}{1}
\tableofcontents
\thispagestyle{empty}

\pagebreak
\section{Introduction}

The study of sample covariance matrices is the oldest random matrix theory, predating Wigner's introduction of the Gaussian ensembles into physics by nearly three decades. Given a sample $X_1,\dots, X_n\in\R^p$ drawn from a large, centred population, form the $p\times n$ data matrix $X = [X_1\dots X_n]$; the $p\times p$ matrix $S = X X^\dag$ plays a central role in multivariate statistical analysis \citep{Muirhead, Bai, Anderson2}.  The distribution in the i.i.d.\ Gaussian case is named after Wishart who computed the density in 1928. The classical story is that of the consistency of the \textbf{sample covariance matrix} $\frac{1}{n}S$ as an estimator of the \textbf{population covariance matrix} $\Sigma = \E X_i X_i^\dag$ when the dimension $p$ is fixed and the sample size $n$ becomes large. The law of large numbers already gives $\frac{1}{n}S\to\Sigma$.  In this fixed dimensional setting, the eigenvalues $\lambda_1\ge\dots\ge\lambda_p$ of $S$ produce consistent estimators of the eigenvalues $\ell_1\ge\dots\ge\ell_p$ of $\Sigma$: for example, the \textbf{sample eigenvalue} $\frac{1}{n}\lambda_k$ tends almost surely to the \textbf{population eigenvalue} $\ell_k$ as $n\to\infty$, with Gaussian fluctuations on the order $n^{-1/2}$ \citep{Anderson}. The same holds in the complex case $X_i\in\C^p$.

Contemporary problems typically involve \textbf{high dimensional data}, meaning that $p$ is large as well---perhaps on the same order as $n$ or even larger.  In this setting, say with \textbf{null covariance} $\Sigma = I$, the sample eigenvalues may no longer concentrate around the population eigenvalue 1 but rather spread out over a certain compact interval. If $p/n\to c$ with $0<c\le 1$, \citet{MP} proved that a.s.\ the empirical spectral distribution $\frac{1}{p}\sum_k\delta_{\lambda_k/n}$ converges weakly to the continuous distribution with density
\[
\frac{\sqrt{(b-x)(x-a)}}{2\pi c x}\1_{[a,b]}(x)
\]
where $a = (1-\sqrt{c})^2$ and $b = (1+\sqrt{c})^2$. (The singular case $c>1$ is similar by the obvious duality between $n$ and $p$, except that the $p-n$ zero eigenvalues become an atom at zero of mass $1-c^{-1}$.) This \textbf{Mar\v{c}enko-Pastur law} is the analogue of Wigner's semicircle law in this setting of multiplicative rather than additive symmetrization \citep[see also][]{SB}. The assumption of Gaussian entries may be significantly relaxed.

Often one is primarily interested in the \emph{largest} eigenvalues, as for example in the widely practiced statistical method of principal components analysis.  Here the goal is a good low-dimensional projection of a high-dimensional data set, i.e.\ one that captures most of the variance; the structure of the significant trends and correlations is estimated using the largest sample eigenvalues and their eigenvectors. The challenge is to determine which observed eigenvalues actually represent structure in the population, and understanding the behaviour in the null case is therefore an essential first step.

In the null case the first-order behaviour is simple: $\frac{1}{n}\lambda_k\to b$ a.s.\ for each fixed $k$ as $n\to\infty$, i.e.\ none have limits beyond the edge of the support of the limiting spectral distribution \citep{Geman, YBK}.  More interestingly, the fluctuations are no longer asymptotically Gaussian but are rather those now recognized as universal at a real symmetric or Hermitian \textbf{random matrix soft edge}: they are on the order $n^{-2/3}$, asymptotically distributed according to the appropriate \textbf{Tracy-Widom law}.   The latter were introduced by \citet{TW1, TW2} as limiting largest eigenvalue distributions for the Gaussian ensembles \citep[see also][]{F1} and have since been found to occur in diverse probabilistic models. The limit theorems for sample covariance matrices were proved by \citet{Jo} in the complex case and by \citet{J1} in the real case \citep[see][for the first universality results here]{Soshnikov}. Restrictions $c\neq 0,\infty$ on the limiting dimensional ratio were removed by \citet{EK1} \citep[see also][]{P2}.

Motivated by principal components analysis, it is natural to study the behaviour of the largest sample eigenvalues when the population covariance is not null but rather has a few trends or correlations.  \citet{J1} proposed the \textbf{spiked population model} in which all but a fixed finite number of population eigenvalues (the \textbf{spikes}) are taken to be 1 as $n,p$ become large.  \citet*{BBP} (\textbf{BBP}) analyzed the spiked \emph{complex} Wishart model and discovered a very interesting phenomenon:  a phase transition in the asymptotic behaviour of the largest sample eigenvalue as a function of the spikes.  We restrict attention to the case of a single spike in the present chapter, setting $\ell_1 = \ell$, $\ell_2 = \ell_3 = \cdots = 1$.

In this \textbf{rank one perturbed case}, \citetalias{BBP} describe three distinct regimes. Assume that $p/n=\gamma^2$ is compactly contained in $(0,1]$. If $\ell_{n,p}$ is in compactly contained in $(0,1+\gamma)$ then the behaviour of the top eigenvalue is exactly the same as in the null case:
\[
\P\left(\tfrac{\gamma^{-1}}{\left(1+\gamma^{-1}\right)^{4/3}} n^{2/3}\left(\tfrac{1}{n}\lambda_1 - \left(1+\gamma\right)^2\right) \le x\right) \to F_2(x),
\]
where $F_2$ is the Tracy-Widom law for the top GUE eigenvalue. This is the \textbf{subcritical regime}. If $\ell_{n,p}$ is compactly contained in $(1+\gamma,\infty)$ then the top eigenvalue separates from the bulk and has Gaussian fluctuations on the order $n^{-1/2}$:
\[
\P\left(\left(\ell^2-\gamma^2\tfrac{\ell^2}{(\ell-1)^2}\right)^{-1/2}n^{1/2}\left(\tfrac{1}{n}\lambda_1 - \left(\ell + \gamma^2\tfrac{\ell}{(\ell-1)}\right)\right)  \le x\right) \to \frac{1}{\sqrt{2\pi}}\int_{-\infty}^x e^{-t^2/2}\,dt.
\]
This is the \textbf{supercritical regime}.  Finally there is a one-parameter family of \textbf{critical} scalings in which $\ell_{n,p} - (1+\gamma)$ is on the order $n^{-1/3}$; these double scaling limits  are tuned so that the fluctuations---which are on the order $n^{-2/3}$ as in the subcritical case---are asymptotically given by a certain one-parameter family of deformations of $F_2$. We refer the reader to the original work for details.  Subsequent work includes a treatment of the singular case $p>n$ along the same lines \citep{Onatski}, deeper investigations into the limiting kernels \citep{DF}, and generalizations beyond the spiked model \citep{EK2} and away from Gaussianity \citep{BY,FP2}. \citetalias{BBP} conjectured a similar phase transition for spiked \emph{real} Wishart matrices, in the sense that all scalings should be the same but the limiting distributions would be different.

Now often referred to as the BBP transition, this picture is relevant in various applications. Within mathematics it has been applied to the TASEP model of interacting particles on the line \citep{BC}. Spiked complex Wishart matrices occur in problems in wireless communications \citep{Telatar}. With these two exceptions, however, most applications involve data that are \emph{real} rather than complex. They include economics and finance---\citet{Harding} used the phase transition to explain an old standard example of the failure of PCA---and medical and population genetics---\citet{Patterson} discuss its role in attempting to answer such questions as ``Given genotype data, is it from a homogeneous population?'' Further applications include speech recognition, statistical learning and the physics of mixtures \citep[see][for references]{J2, Paul, FP2}.  In general, asymptotic distributions in the non-null cases are relevant when evaluating the power of a statistical test \citep{J2}.

Despite these developments, the conjectured BBP picture for spiked real Wishart matrices has proven elusive even in the rank one case. The difficulty is with the joint eigenvalue density: The complex case involves an integral over the unitary group that \citetalias{BBP} analyzed via the Harish-Chandra-Itzykson-Zuber integral, a tool originating in representation theory that appears to have no straightforward analogue over the orthogonal group. Much is known, however. At the level of a law of large numbers, the phase transition is described by \citet{BS}; a related separation phenomenon was observed already by \citet{BaiS1,BaiS2}. A broad generalization of the results on a.s.\ limits is developed by \citet{BN} and dubbed ``spiked free probability theory".  \cite{Paul, BY} prove Gaussian central limit theorems in the supercritical regime. \citet{FP2} prove Tracy-Widom fluctuations in the subcritical regime under the scaling assumptions of \citetalias{BBP}.   Interestingly, \cite{W} obtained a critical limiting distribution for certain rank one spiked \emph{quaternion} Wishart matrices.

It remains to obtain the asymptotic behaviour in the critically spiked regime around the phase transition in the real case.  We do so here, establishing the existence of limiting distributions under the scalings conjectured by \citetalias{BBP} and characterizing the laws. Our results apply also to the complex case, and they are more general than the corresponding statements from \citetalias{BBP}. We do not restrict the scaling of $n, p$ beyond requiring that they tend to infinity together, nor that of $\ell$ beyond what is strictly necessary for the existence of a limiting distribution in the subcritical or critical regimes.  We therefore allow for certain relevant possibilities that were previously excluded, namely $p\ll n$ and $p\gg n$. The picture of the dependence on the spike is also more complete: we include all intermediate scalings of $\ell$ with $n,p$ across the subcritical and critical regimes. Separately, we describe a joint convergence in law when the same underlying data is spiked with different $\ell$.

Since this article was first posted, \citet{Mo} gave a different treatment of the real rank one case.  Despite the difficulties mentioned, he succeeds with the standard program of obtaining forms for the joint eigenvalue and largest eigenvalue distributions and doing asymptotic analysis on the latter. His description of the limiting distribution naturally looks very different from ours.  See \cite{F3} for some remarks on the two treatments and an alternative construction of the ``general $\beta$'' model we now introduce.

We bypass the eigenvalue density altogether; our starting point is rather a \emph{reduction of the matrix to tridiagonal form via Householder's algorithm}, a well-known tool in numerical analysis. \cite{Trotter} observed that the algorithm interacts nicely with the Gaussian structure, using the resulting forms to derive the Wigner semicircle and Mar\v{c}enko-Pastur laws without going through their moments.  Observing the similarity of the forms in the $\beta=1,2,4$ cases, \citet{DE} introduced interpolating matrix ensembles for all $\beta>0$ whose eigenvalue density is given by Dyson's \textbf{Coulomb} or \textbf{log gas model}
\begin{equation}\label{loggas}
\frac{1}{Z}\prod_{j< k}\abs{\lambda_j-\lambda_k}^\beta\prod_j v(\lambda_j)^{\beta/2}
\end{equation}
where $v$ is the Hermite or the Laguerre weight and $Z$ is a normalizing factor \citep[see][for more on such models]{F2}.  Incidentally, Trotter's argument applies to these \textbf{general $\beta$ analogues} and establishes Wigner semicircle and Mar\v{c}enko-Pastur laws in this setting.  An extension to more general weights is part of a forthcoming work of \citet{KRV}.

The second step is to consider the tridiagonal ensemble as a discrete random Schr\"odinger operator (i.e.\ discrete Laplacian plus random potential) and then take a scaling limit at the soft edge to obtain a certain continuum random Schr\"odinger operator on the half-line.  This ``stochastic operator approach to random matrix theory'' was pioneered by \citet{ES}, \citet{S}; in the soft edge case their heuristics were proved by \citet{RRV}, who in particular established joint convergence of the largest eigenvalues. Our method is directly based on the latter work and we refer to it throughout by the initials \textbf{RRV}.  The key point is that both steps can be adapted to the setting of rank one perturbations.  As we will see, the limiting operator feels the perturbation in the boundary condition at the origin.

In detail, let $X$ be a $p\times n$ sample matrix whose columns are independent real $N(0,\Sigma)$ with $\Sigma = \diag\bigl(\ell,1,\dots,1)$ for some $\ell>0$; we shall say $S = XX^\dag$ has the \textbf{$\ell$-spiked $p$-variate real Wishart distribution with $n$ degrees of freedom}. (There is no loss of generality in taking $\Sigma$ diagonal in the Gaussian case.) We also consider the complex and quaternion cases. The tridiagonalization is carried out in detail in Section~\ref{s.3}. The result is a symmetric tridiagonal $(n\wedge p)\times(n\wedge p)$ matrix $W^\dag W$, where $W$ is a certain bidiagonal matrix with the same nonzero singular values as $X$. Explicitly, $W$ is given by
\begin{equation}\label{Wb}
W^{\beta,\ell}_{n, p} = \frac{1}{\sqrt\beta}\begin{bmatrix}
\sqrt{\ell}\,\wt\chi_{\beta n}
\\ \chi_{\beta(p-1)} & \wt\chi_{\beta(n-1)}
\\ & \chi_{\beta(p-2)} & \wt\chi_{\beta(n-2)}
\\ & & \ddots & \ddots
\\ & & & \chi_{\beta(p-(n\wedge p)+1)} & \wt\chi_{\beta(n-(n\wedge p)+1)}
\\ & & & & \chi_{\beta(p-(n\wedge p))}
\end{bmatrix}
\end{equation}
where $\beta = 1,2,4$ in the real, complex and quaternion cases respectively and the $\chi,\wt\chi$'s are mutually independent chi distributed random variables with parameters given by their indices. In fact~\eqref{Wb} makes sense for any $\beta>0$, and the resulting ensemble $W^\dag W$ is a ``spiked version'' of the $\beta$-Laguerre ensemble of \citet{DE}; we call it the \textbf{$\ell$-spiked $\beta$-Laguerre ensemble with parameters $n,p$}. Such a matrix almost surely has exactly $n\wedge p$ distinct nonzero eigenvalues by the theory of Jacobi matrices. In the null case $\ell=1$, their joint density is~\eqref{loggas} with the Laguerre weight $v(x)=x^{\abs{n-p}+1-2/\beta}e^{-x}\1_{x> 0}$. We note that there is an obvious coupling of~\eqref{Wb} over all $\ell>0$; in the spiked Wishart cases it corresponds to the natural coupling obtained by considering $X$ as a matrix of standard Gaussians left multiplied by $\sqrt{\Sigma}$.

In order to state our results, we now recall the \textbf{stochastic Airy operator} introduced by \citet{ES}. Formally this is the random Schr\"odinger operator
\[
\Hc_\beta = -\frac{d^2}{dx^2} + x + \trb b_x'
\]
acting on $L^2(\R_+)$ where $b_x'$ is standard Gaussian white noise. \citetalias{RRV} defined this operator rigorously and considered the eigenvalue problem $\Hc_\beta f = \Lambda f$ with Dirichlet boundary condition $f(0) = 0$. We will consider a general homogeneous boundary condition $f'(0) = wf(0)$, a Neumann or Robin condition for $w\in(-\infty,\infty)$ with the limiting Dirichlet case naturally corresponding to $w=+\infty$. Precise definitions will be given in Section~\ref{s.2} in a more general setting; for now, we write $\Hc_{\beta,w}$ to indicate the stochastic Airy operator together with this boundary condition.

We will see that, almost surely, $\Hc_{\beta,w}$ is bounded below with purely discrete, simple spectrum $\{\Lambda_0<\Lambda_1<\cdots\}$ for all $w\in(-\infty,\infty]$. This fact will be established simultaneously with the standard variational characterization: in Proposition~\ref{p.var}, we show in particular that $\Lambda_k$ and the corresponding eigenfunction $f_k$ are given recursively by
\begin{equation}\label{SAvar}
\Lambda_k \,=\, \inf_{\substack{f\in L^2,\ \norm{f}=1,\\ f\perp f_0,\dots, f_{k-1}}} \int_0^\infty \bigl({f'(x)}^2 + xf^2(x)\bigr)dx + wf(0)^2 + \trb\int_0^\infty f^2(x)\,db_x
\end{equation}
in which we consider only candidates $f$ for which the first integral is finite, and the stochastic integral is defined pathwise via integration by parts. Recall from \citetalias{RRV} that the distribution $F_{\beta,\infty}$ of $-\Lambda_0$ in the Dirichlet case $w = +\infty$ may be taken as a definition of \textbf{Tracy-Widom($\beta$)} for general $\beta>0$, a one-parameter family of distributions interpolating between those at the standard values $\beta= 1,2,4$.  Fixing $\beta$, the distributions $F_{\beta,w}$ for finite $w$ may be thought of as a family of deformations of Tracy-Widom($\beta$).  We note that the pathwise dependence of $\Hc_{\beta,w}$ on the Brownian motion allows the operators to be coupled over $w$ in a natural way.

Our first result gives a convergence in distribution at the soft edge of the $\ell$-spiked $\beta$-Laguerre spectrum over the full range of subcritical and critical scalings. Note the absence of extraneous hypotheses on $n$, $p$ and $\ell_{n,p}$.

\begin{theorem}\label{t.W} Let $\ell_{n,p}>0$. Let $S = S_{n,p}$ have the real (resp.\ complex, quaternion) $\ell_{n,p}$-spiked $p$-variate Wishart distribution with $n$ degrees of freedom and set $\beta = 1$ (resp.\ $2$, $4$), or, let $\beta >0$ and take $S_{n,p}$ from the $\ell_{n,p}$-spiked $\beta$-Laguerre ensemble with parameters $n,p$.  Writing $m_{n,p} = \left(n^{-1/2} + p^{-1/2}\right)^{-2/3}$, suppose that
\begin{equation}\label{Wspike}
m_{n,p}\left(1-\sqrt{n/p}\bigl(\ell_{n,p}-1\bigr)\right)\,\to\,w\in(-\infty,\infty]\qquad\text{as }n\wedge p\to\infty.
\end{equation}
Let $\lambda_1>\dots>\lambda_{n\wedge p}$ be the nonzero eigenvalues of $S$. Then, jointly for $k=1,2,\ldots$ in the sense of finite-dimensional distributions, we have
\[
\frac{m_{n,p}^2}{\sqrt{np}}\left(\lambda_k - \left(\sqrt{n}+\sqrt{p}\right)^2\right)\,\Rightarrow\, -\Lambda_{k-1}\qquad\text{as }n\wedge p\to\infty
\]
where $\Lambda_0<\Lambda_1<\cdots$ are the eigenvalues of $\Hc_{\beta,w}$. Furthermore, the convergence holds jointly with respect to the natural couplings over all $\{\ell_{n,p}\},w$ satisfying~\eqref{Wspike}.
\end{theorem}

\begin{remark}
In the tridiagonal basis, the convergence holds also at the level of the corresponding eigenvectors. If the eigenvector corresponding to $\lambda_k$ is embedded in $L^2(\R_+)$ as a step-function with step width $m_{n,p}^{-1}$ and support $[0,({n\wedge p})/m_{n,p}]$, then it converges to $f_{k-1}$ in distribution with respect to the $L^2$ norm; the details are the subject of the next section. In particular, distributional convergence of the rescaled tridiagonal operators to $\Hc_{\beta,w}$ holds in the norm resolvent sense \citep[see e.g.][]{Weidmann}.  Defining $\Hc_{\beta,w}$ as a closed operator on the appropriate (random) dense subspace of $L^2$ requires some care, however \citep[see e.g.][]{SS} and we shall not pursue it here.
\end{remark}

\begin{remark}
The supercritical regime $w = -\infty$ sees a macroscopic separation of the largest eigenvalue from the bulk of the spectrum; the fluctuations of $\lambda_1$ are on a larger order and they are asymptotically Gaussian, independent of the rest.  Though known for real and complex spiked sample covariance matrices \citep[BBP,][]{Paul, BY}, existing results do not cover intermediate ``vanishingly supercritical'' scalings of $\ell$ with $n,p$ and thus leave a certain gap between the critical and supercritical regimes. This gap can be addressed using the stochastic Airy framework \citep{Super}.
\end{remark}

\begin{remark}
Work of \citet{FP2} immediately allows extension of the previous theorem in the real and complex spiked Wishart cases to more general real and complex spiked sample covariance matrices. More precisely, the i.i.d.\ multivariate Gaussian columns of the data matrix $X$ may be replaced with i.i.d.\ columns having zero mean and rank one spiked diagonal covariance, and satisfying some moment conditions. These authors make the same assumptions on the dimension ratio as \citetalias{BBP}, but the null case universality result of \citet{P2} suggest these could be removed.
\end{remark}

We prove Theorem~\ref{t.W} by establishing a more general technical result, Theorem~\ref{t.conv} in Section~\ref{s.2}. The latter theorem gives conditions under which the low-lying eigenvalues and corresponding eigenvectors of a large random symmetric tridiagonal matrix converge in law to those of a random Schr\"odinger operator on the half-line with a given potential and homogeneous boundary condition at the origin. Verifying the hypotheses for suitably scaled spiked Laguerre matrices will be relatively straightforward; we do it in Section~\ref{s.3}. The approach follows that of \citetalias{RRV}, where the null case of Theorem~\ref{t.W} is treated.

One advantage of such an approach is that it immediately yields results for other matrix models as well.  In particular, finite-rank additive perturbations of \textbf{Gaussian orthogonal, unitary and symplectic ensembles (GO/U/SE)} have received considerable attention. The analogue of the BBP theorem in the perturbed GUE setting was established by \citet{P1, DF}. \citet{BFF} treat an interesting generalization and mention some applications to physics.  We consider a simple additive rank one perturbation of the GOE obtained by shifting the mean of every entry by the same constant $\mu/\sqrt{n}$. By orthogonal invariance, this has the same effect on the spectrum as shifting the (1,1) entry by $\sqrt{n}\,\mu$. With this perturbation, the usual tridiagonalization procedure works; the resulting form is the $\beta=1$ case of
\begin{equation}\label{Gb}
G^{\beta,\mu}_{n} = \frac{1}{\sqrt\beta}\begin{bmatrix}
\sqrt{2}\,g_1 + \sqrt{\beta n}\,\mu & \chi_{\beta(n-1)}
\\ \chi_{\beta(n-1)} &  \sqrt{2}\,g_2 & \chi_{\beta(n-2)}
\\ & \chi_{\beta(n-2)} & \sqrt{2}\,g_3 & \ddots
\\ & & \ddots & \ddots & \chi_\beta
\\ & & & \chi_{\beta} &  \sqrt{2}\,g_n
\end{bmatrix},
\end{equation}
where the $ g$'s are independent standard Gaussians and the $\chi$'s are independent Chi random variables indexed by their parameter as before. The analogous procedure for a shifted mean GUE (resp.\ GSE) yields~\eqref{Gb} with $\beta =2$ (resp.\ 4). This matrix ensemble is a perturbed version of the \textbf{$\beta$-Hermite ensemble} of \citet{DE}. In the unperturbed case $\mu=0$, the joint eigenvalue density is~\eqref{loggas} with the Hermite weight $v(x)=e^{-x^2/2}$. Again, the models are naturally coupled over all $\mu\in\R$.

As in the spiked real Wishart setting, the critical regime for the rank one perturbed GOE has resisted description. We show that the phase transition in the perturbed Hermite ensemble has the same characterization as the one in the Laguerre ensemble.

\begin{theorem}\label{t.G} Let $\mu_n\in\R$. Let $G = G_n$ be a $(\mu_n/\sqrt{n})$-shifted mean $n\times n$ GOE (resp.\ GUE, GSE) matrix and set $\beta = 1$ (resp.\ $2$, $4$), or, let $\beta >0$ and take $G_n = G_n^{\beta,\mu_n}$ as in $\eqref{Gb}$. Suppose that
\begin{equation}\label{Gspike}
n^{1/3}\left(1-\mu_n\right)\,\to\,w\in(-\infty,\infty]\qquad\text{as }n\to\infty.
\end{equation}
Let $\lambda_1>\dots>\lambda_n$ be the eigenvalues of $G$. Then, jointly for $k=0,1,\ldots$ in the sense of finite-dimensional distributions, we have
\[
n^{1/6}\left(\lambda_k - 2\sqrt{n}\right)\,\Rightarrow\, -\Lambda_{k-1}\qquad\text{as }n\to\infty
\]
where $\Lambda_0<\Lambda_1<\cdots$ are the eigenvalues of $\Hc_{\beta,w}$. Furthermore, the convergence holds jointly with respect to the natural couplings over all $\{\mu_n\},w$ satisfying~\eqref{Gspike}.
\end{theorem}

\begin{remark} The remarks following the previous theorem apply also to this theorem; the universality issue is discussed in \citet{FP1}.
\end{remark} 
The limit of a rank one perturbed general $\beta$ soft edge thus seems to be universal, just as at $\beta=2$.  We offer two alternative descriptions.

\begin{theorem}\label{t.char} Fix $\beta>0$ and let $\Lambda_0$ be the ground state energy of $\Hc_{\beta,w}$ where $w\in(-\infty,\infty]$. The distribution $F_{\beta,w}(x) = \P_{\beta,w}(-\Lambda_0\le x)$ has the following alternative characterizations.
\begin{enumerate}[(i)]

\item(\citetalias{RRV}) Consider the stochastic differential equation
\begin{equation}\label{SDE}
dp_x = \trb db_x + \left(x-p_x^2\right)dx
\end{equation}
and let $\P_{(x_0,w)}$ be the It\=o diffusion measure on paths $\{p_x\}_{x\ge x_0}$ started from $p_{x_0} = w$. A path almost surely either explodes to $-\infty$ in finite time or grows like $p_x\sim\sqrt{x}$ as $x\to\infty$, and we have
\vspace{-6pt}
\begin{equation}\label{diffusion}
F_{\beta,w}(x) = \P_{(x,w)}\bigl(\text{\emph{$p$ does not explode}}\bigr).
\end{equation}

\item The boundary value problem
\begin{gather}\label{PDE}
\frac{\del F}{\del x} + \frac{2}{\beta}\frac{\del^2 F}{\del w^2} + \bigl(x-w^2\bigr)\frac{\del F}{\del w} = 0\qquad\text{ for }(x,w)\in\R^2,
\\\label{BC}\begin{aligned}
F(x,w)\to 1\qquad&\text{ as }x,w\to\infty\text{ together},
\\F(x,w)\to 0\qquad&\text{ as }w\to-\infty\text{ with }x\text{ bounded above}
\end{aligned}
\end{gather}
has a unique bounded solution, and we have $F_{\beta,w}(x) = F(x,w)$ for $w\in(-\infty,\infty)$. We recover the Tracy-Widom$(\beta)$ distribution $F_{\beta,\infty}(x) = \lim_{w\to\infty} F(x,w)$.
\end{enumerate}
\end{theorem}

\begin{remark} These characterizations can be extended to the higher eigenvalues; details appear in Section~\ref{s.4}.
\end{remark}

In \citetalias{RRV} the diffusion characterization is derived with classical tools, namely the Riccati transformation and Sturm oscillation theory. We review the relevant facts in Section~\ref{s.4} before proceeding to the boundary value problem.  While the latter characterization amounts to a straightforward reformulation of the former, it is appealing in that it involves no stochastic objects. It also turns out to offer a good way of evaluating the distributions numerically \citep{S2}.  Most interestingly, however, it provides a sought-after connection with known integrable structure at $\beta = 2,4$.

To wit, let $u(x)$ be the \textbf{Hastings-McLeod solution} of the \textbf{homogeneous Painlev\'e II equation}
\begin{equation}\label{PII}
u'' = 2u^3 + xu,
\end{equation}
characterized by
\begin{equation}\label{HM}
u(x)\sim\Ai(x)\quad\text{as }x\to +\infty
\end{equation}
where $\Ai(x)$ is the Airy function (characterized in turn by $\Ai'' = x\Ai$ and $\Ai(+\infty)=0$); it is known that there is a unique such function and that it has no singularities on $\R$ \citep{HM}. Put
\vspace{-6pt}
\begin{gather}\label{v}{\textstyle
v(x) = \int_x^\infty u^2},
\\\label{EF}{\textstyle
E(x) = \exp\bigl(-\int_x^\infty u\bigr),\qquad F(x) = \exp\bigl(-\int_x^\infty v\bigr).}
\end{gather}
Next define two functions $f(x,w)$, $g(x,w)$ on $\R^2$, analytic in $w$ for each fixed $x$, by the first order linear ODEs
\begin{equation}\label{w_lax}
\frac{\del}{\del w}\begin{pmatrix}f\\g\end{pmatrix}=\begin{pmatrix}u^2&-wu-u'\\-wu+u'&w^2-x-u^2\end{pmatrix}\begin{pmatrix}f\\g\end{pmatrix}
\end{equation}
and the initial conditions
\begin{equation}\label{IC}
f(x,0) \,=\, E(x) \,=\, g(x,0).
\end{equation}

Equation \eqref{w_lax} is one member of the Lax pair for the Painlev\'e II equation.  The functions $f,g$ can also be defined in terms of the solution of the associated Riemann-Hilbert problem; analysis of the latter yields some information about $u,f,g$ summarized in Facts \ref{f.p1} and \ref{f.p2} below. The following theorem expresses the relationship between the objects just defined and the general $\beta$ characterization at $\beta=2,4$. The proof is given in Section~\ref{s.5}.

\begin{theorem}\label{t.id}
The identities
\begin{align}\label{id2}
F_{2,w}(x) &= f(x,w)F(x),
\\\label{id4}
F_{4,w}(x) &= \left.\left(\frac{(f+g)E^{-1/2}+(f-g)E^{1/2}}{2}\right)F^{1/2}\right|_{(2^{2/3}x,\,2^{1/3}w)}
\end{align}
hold and follow directly from Theorem \ref{t.char} and Facts \ref{f.p1} and \ref{f.p2}.
\end{theorem}

The formula for $F_{2,w}$ is given by \citet{B}, although it appeared earlier in work of \citet{BR1, BR2} in a very different context.  The formula for $F_{4,w}$ appears in \citet{BR1, BR2} in a disguised form; the $w=0$ case is obtained by \citet{W}, but it is a new result in this context for $w\neq 0,\infty$. In the $\beta=4$ case we thus use our characterization to prove a guess.

In particular, we recover the Painlev\'e II representations of Tracy and Widom at these $\beta$ in a novel and simple way.

\begin{corollary}[\citealp{TW1,TW2}, \citetalias{BBP} 2005, \citealp{W}]\label{c.id}
We have
\begin{align}
F_{2,\infty}(x) &= F(x),
\\
F_{4,\infty}(2^{-2/3}x) &= \half\bigl(E^{1/2}(x)+E^{-1/2}(x)\bigr)F^{1/2}(x),
\\
F_{2,0}^{1/2}(x) = F_{4,0}(2^{-2/3}x) &= E^{1/2}(x)F^{1/2}(x).
\end{align}
\end{corollary}

\begin{remark}
The latter distribution is known to be $F_{1,\infty}(x)$ \citep{TW2}. Unfortunately we lack an independent proof.
\end{remark}

A number of points remain somewhat mysterious.  Most obviously, we lack a connection in the $\beta = 1$ case; while the literature previously did not even suggest a guess, it would now be illuminating to reconcile \eqref{PDE}, \eqref{BC} with the formula obtained by~\citet{Mo}.  Even at $\beta = 2, 4$ it seems there should be a more direct way to derive or at least understand the connection.  From the point of view of the PDE \eqref{PDE}, some kind of extra structure appears to be present at certain special values of the parameter $\beta$; what about other values?  From the point of view of nonlinear special functions, we have shown directly---independently of any limit theorems---how the well-studied Hastings-McLeod solution admits characterization in terms of a simple linear parabolic boundary value problem in the plane.

We close this introduction by advertising the sequel, in which we treat the general spiked model with analogous methods.

\section{The limit of a spiked tridiagonal ensemble}\label{s.2}

In this section we strengthen the argument of \citetalias{RRV} to apply in the rank one spiked cases.  The main convergence result will be applied in the next section to the tridiagonal forms described in the introduction.

Theorem~\ref{t.conv} below generalizes Theorem 5.1 of \citetalias{RRV} in a natural way, giving conditions under which the low-lying eigenvalues and corresponding eigenvectors of a random symmetric tridiagonal matrix converge in law to those of a random Schr\"odinger operator on the half-line with a given potential \emph{and homogeneous boundary condition at the origin}.  We include substantial parts of the original argument both for completeness and to highlight the new material; see \cite{AGZ} for another presentation of the original argument in a special case.


\subsection{Matrix model and embedding}

Underlying the convergence is the embedding of the discrete half-line $\Z_+ = \{0,1,\ldots\}$ into $\R_+ = [0,\infty)$ via $j\mapsto j/m_n$, where the scale factors $m_n\to\infty$ but with $m_n = o(n)$. Define an associated embedding of function spaces by step functions:
\[
\ell^2_n(\Z_+)\hookrightarrow L^2(\R_+),\quad (v_0,v_1,\ldots)\,\mapsto\, v(x) = v_{\lfloor m_n x\rfloor},
\]
which is isometric with $\ell^2_n$-norm $\norm{v}^2 = m_n^{-1}\sum_{j=0}^\infty v_j^2$.  Identify $\R^n$ with the initial coordinate subspace $\{v\in\ell^2_n:v_j=0, j\ge n\}$. We will generally not refer to the embedding explicitly.

We define some operators on $L^2$, all of which leave $\ell^2_n$ invariant. The translation operator $(T_n f)(x) = f(x+m_n^{-1})$ extends the left shift on $\ell^2_n$. The difference quotient $D_n = m_n(T_n - 1)$ extends a discrete derivative. Write $E_n = \diag(m_n,0,0,\ldots)$ for multiplication by $m_n\1_{[0,m_n^{-1})}$, a ``discrete delta function at the origin'', and $R_n = \diag(1,\dots,1,0,0,\ldots)$ for multiplication by $\1_{[0,n/m_n)}$, which extends orthogonal projection $\ell^2_n\to\R^n$.

Let $(y_{n,i;j})_{j=0,\ldots,n}$, $i=1,2$ be two discrete-time real-valued random processes with $y_{n,i;0} = 0$, and let $w_n$ be a real-valued random variable. Embed the processes as above. 
Define a ``potential'' matrix (or operator)
\[
V_n = \diag(D_n y_{n,1}) + \half\bigl(\diag(D_n y_{n,2})T_n + T_n^\dag\diag(D_n y_{n,2})\bigr),
\]
and finally set
\begin{equation}\label{Hn}
H_n \,=\, R_n\bigl(D_n^\dag D_n + V_n + w_n E_n\bigr).
\end{equation}
This operator leaves the subspace $\R^n$ invariant. The matrix of its restriction with respect to the coordinate basis is symmetric tridiagonal, with on- and off-diagonal processes
\begin{gather}
\begin{split}\label{Hn1}
m_n^2+(y_{n,1;1}+w_n)m_n,\ 2m_n^2+(y_{n,1;2}-y_{n,1;1})m_n,\ \dots,\qquad\qquad\qquad\qquad\qquad\\2m_n^2+(y_{n,1;n}-y_{n,1;n-1})m_n
\end{split}
\\\begin{split}\label{Hn2}
-m_n^2+\half y_{n,2;1}m_n,\ -m_n^2+\half(y_{n,2;2}-y_{n,2;1})m_n,\ \dots,\qquad\qquad\qquad\qquad\qquad\qquad\\-m_n^2+\half(y_{n,2;n-1}-y_{n,2;n-2})m_n
\end{split}
\end{gather}
respectively. We denote this random matrix also as $H_n$, and call it a \textbf{spiked tridiagonal ensemble}. (We could have absorbed $w_n$ into $y_{n,1}$ as an additive constant, but keep it separate for reasons that will soon be apparent.)

As in \citetalias{RRV}, convergence rests on a few key assumptions on the random variables just introduced. By choice, no additional scalings will be required. 

\bigskip\noindent\emph{Assumption 1 (Tightness and convergence).}\quad There exists a continuous random process $\{y(x)\}_{x\ge 0}$ with $y(0)=0$ such that
\begin{equation}\label{a1}\begin{gathered}
\{y_{n,i}(x)\}_{x\ge 0},\ i=1,2\quad \text{are tight in law,}
\\ y_{n,1}+y_{n,2}\,\Rightarrow\, y\quad\text{in law}
\end{gathered}\end{equation}
with respect to the compact-uniform topology on paths.

\bigskip\noindent\emph{Assumption 2 (Growth and oscillation bounds).}\quad There is a decomposition
\begin{equation*}
y_{n,i;j} = m_n^{-1}\sum_{k=0}^{j-1}\eta_{n,i;k}+\omega_{n,i;j}
\end{equation*}
with $\eta_{n,i;j}\ge 0$ such that for some deterministic unbounded nondecreasing continuous functions $\ol\eta(x)>0$, $\zeta(x)\ge 1$ not depending on $n$, and random constants $\kappa_n\ge 1$ defined on the same probability spaces, the following hold: The $\kappa_n$ are tight in distribution, and for each $n$ we have almost surely
\begin{align}\label{a21}
\ol\eta(x)/\kappa_n-\kappa_n\:\le\:\eta_{n,1}(x)+\eta_{n,2}(x)\:&\le\:\kappa_n\bigl(1+\ol\eta(x)\bigr),
\\\label{a22} \eta_{n,2}(x)\:&\le\:2m_n^2,
\\\label{a23} \abs{\omega_{n,1}(\xi)-\omega_{n,1}(x)}^2 +\abs{\omega_{n,2}(\xi)-\omega_{n,2}(x)}^2\:&\le\:\kappa_n\bigl(1+\ol\eta(x)/\zeta(x)\bigr)
\end{align}
for all $x,\xi\in[0,n/m_n]$ with $\abs{\xi-x}\le 1$.

\bigskip\noindent\emph{Assumption 3 (Critical or subcritical spiking).} For some nonrandom $w\in(-\infty,\infty]$, we have
\begin{equation}\label{a3}
w_n\,\to\,w\quad\text{in probability}.
\end{equation}

The necessity of first and third assumptions will be evident when we define a continuum limit and prove convergence.  The more technical second assumption ensures tightness of the matrix eigenvalues; its limiting version (derived in the next subsection) will guarantee discreteness of the limiting spectrum. Lastly, we note that for given $y_n$ the models may be coupled over different choices of $w_n$.

\subsection{Reduction to deterministic setting}

In the next subsection we will define a limiting object in terms of $y$ and $w$; we want to prove that the discrete models converge to this continuum limit in law. We reduce the problem to a deterministic convergence statement as follows. First, select any subsequence.  It will be convenient to extract a further subsequence so that certain additional tight sequences converge jointly in law; Skorokhod's representation theorem \citep[see][]{EthierKurtz} says this convergence can be realized almost surely on a single probability space.  We may then proceed pathwise.

In detail, consider \eqref{a1}--\eqref{a3}. Note in particular that the upper bound of~\eqref{a21} shows that the piecewise linear process $\left\{\int_0^x\eta_{n,i}\right\}_{x\ge0}$ is tight in distribution under the compact-uniform topology for $i=1,2$. Given a subsequence, we pass to a further subsequence so that the following distributional limits exist jointly: 
\begin{equation}\label{ar}\begin{aligned}
y_{n,i} \,&\Rightarrow\, y_i,
\\ {\textstyle\int_0\eta_{n,i}}\,&\Rightarrow\, \eta_i^\dag,
\\ \kappa_n \,&\Rightarrow\, \kappa,
\end{aligned}\end{equation}
for $i=1,2$, where convergence in the first two lines is in the compact-uniform topology.  We realize~\eqref{ar} pathwise a.s.\ on some probability space and continue in this deterministic setting.

We can take the bounds~\eqref{a21},\eqref{a23} to hold with $\kappa_n$ replaced with a single constant $\kappa$. Observe that~\eqref{a21} gives a local Lipschitz bound on the $\int\eta_{n,i}$, which is inherited by their limits $\eta^\dag_i$. Thus $\eta_i = \bigl(\eta^\dag_i\bigr)'$ is defined almost everywhere on $\R_+$, satisfies~\eqref{a21}, and may be defined to satisfy this inequality everywhere. Furthermore, one easily checks that $m_n^{-1}\sum\eta_{n,i}\to\int\eta_i$ compact-uniformly as well (use continuity of the limit). Therefore $\omega_{n,i} = y_{n,i}-m_n^{-1}\sum\eta_{n,i}$ must have a continuous limit $\omega_i$ for $i = 1,2$; moreover, the bound \eqref{a23} is inherited by the limits. Lastly, put $\eta = \eta_1+\eta_2$, $\omega = \omega_1+ \omega_2$ and note that $y_i = \int\eta_i + \omega_i$ and $y = \int\eta + \omega$.

Without further reference to the subsequences, we will assume this situation for the remainder of the section.

\subsection{Limiting operator and variational characterization}

Formally, the limit of the spiked tridiagonal ensemble $H_n$ will be the eigenvalue problem
\begin{equation}\label{e_prob}\begin{gathered}
\quad\Hc f = \Lambda f\quad\text{on }\R_+
\\f'(0) = w f(0),\qquad\qquad f(+\infty)= 0
\end{gathered}\end{equation}
where $\Hc = -d^2/dx^2 + y'(x)$ and $w\in(-\infty,\infty]$ is fixed. If $w=+\infty$, the boundary condition is to be interpreted as $f(0) = 0$; we refer to this as the \textbf{Dirichlet case}, and it will require special treatment in what follows. The primary object for us will be a symmetric bilinear form associated with the eigenvalue problem~\eqref{e_prob}.

Define a space of test functions $C_0^\infty$ consisting of smooth functions on $\R_+$ with compact support that \emph{may contain the origin, except in the Dirichlet case.} Denote by $\norm{\cdot}$ and $\ip{\cdot,\cdot}$ the norm and inner product of $L^2[0,\infty)$. Define a weighted Sobolev norm by
\[
\norm{f}_*^2 \,=\, \Norm{f'}^2 + \Norm{f\weight}^2
\]
and an associated Hilbert space $L^*$ as the closure of $C_0^\infty$ under this norm. Note that our $L^*$ differs slightly from the one in \citetalias{RRV}\@. We register some basic facts about $L^*$ functions.

\begin{fact}\label{f.1}
Any $f\in L^*$ is uniformly H\"older(1/2)-continuous, satisfies $\abs{f(x)}\le \norm{f}_*$ for all $x$, and in the Dirichlet case has $f(0) = 0$.
\end{fact}

\begin{proof}
We have $\abs{f(y)-f(x)} = \abs{\int_x^y f'} \le\norm{f'}\abs{y-x}^{1/2}$. For $f\in C_0^\infty$ we have $f(x)^2 = -\int_x^\infty {(f^2)}' \le 2\norm{f'}\norm{f}\le\norm{f}_*^2$; an $L^*$-bounded sequence in $C_0^\infty$ therefore has a compact-uniformly convergent subsequence, so we can extend this bound to $f\in L^*$ and conclude further that $f(0) = 0$ in the Dirichlet case.
\end{proof}

For future reference, we also record some compactness properties of the $L^*$-norm.

\begin{fact}\label{f.2} Every $L^*$-bounded sequence has a subsequence converging in the following modes: (i) weakly in $L^*$, (ii) derivatives weakly in $L^2$, (iii) uniformly on compacts, and (iv) in $L^2$.
\end{fact}

\begin{proof} (i) and (ii) are just Banach-Alaoglu; (iii) is the previous fact and Arzel\`a-Ascoli again; (iii) implies $L^2$ convergence locally, while the uniform bound on $\int\ol\eta f_n^2$ produces the uniform integrability required for (iv). Note that the weak limit in (ii) really is the derivative of the limit function, as one can see by integrating against functions $\1_{[0,x]}$ and using pointwise convergence.
\end{proof}

We introduce a symmetric bilinear form on $C_0^\infty\times C_0^\infty$ by
\begin{equation}\label{Hya}
\Hyw{\ph,\psi} \,=\, \ip{\ph',\psi'} -\ip{{(\phi\psi)}',y}+ w\,\ph(0)\psi(0),
\end{equation}
dropping the last term in the Dirichlet case. (We could have absorbed $w$ into $y$ as an additive constant in the finite case, but prefer to keep the boundary term separate.) Formally, $\Hyw{\ph,f}$ is just $\ip{\ph,\Hc f}$; notice how the mixed boundary condition is built ``implicitly'' into the form, while the Dirichlet boundary condition is built ``explicitly'' into the space.

\begin{lemma}\label{l.cbound} There are constants $c,C>0$ so that the following bounds holds for all $f\in C_0^\infty$:
\begin{equation}\label{cbound}
c\norm{f}_*^2 - C\norm{f}^2\,\le\,\Hyw{f,f}\,\le\,C\norm{f}_*^2.
\end{equation}
In particular, $\Hyw{\cdot,\cdot}$ extends uniquely to a continuous symmetric bilinear form on $L^*\times L^*$ satisfying the same bounds.
\end{lemma}

\begin{proof}
For the first two terms of~\eqref{Hya}, we use the decomposition $y=\int\eta+\omega$ from the previous subsection.  Integrating the $\int\eta$ term by parts, the limiting version of~\eqref{a21} easily yields
\[
\tfrac{1}{\kappa}\norm{f}_*^2 - C'\norm{f}^2\,\le\,\norm{f'}^2 + \ip{f^2,\eta}\,\le\,\kappa\norm{f}_*^2.
\]
Break up the $\omega$ term as follows. The moving average $\ol \omega_x = \int_x^{x+1}\omega$ is differentiable with $\ol \omega_x' = \omega_{x+1}-\omega_{x}$; writing $\omega = \ol \omega + (\omega - \ol \omega)$, we have
\[
-\Ip{{(f^2)}',\omega} = \Ip{f,\ol \omega'f} + 2\Ip{f',(\ol \omega-\omega)f}.
\]
The limiting version of~\eqref{a23} gives $\max\bigl(\abs{\omega_{\xi}-\omega_{x}},\abs{\omega_{\xi}-\omega_x}^2\bigr)\le C_\e+\e\ol\eta(x)$ for $\abs{\xi-x}\le 1$, where $\e$ can be made small. In particular, the first term above is bounded absolutely by $\e\norm{f}_*^2+C_\e\norm{f}^2$. Averaging, we also get $\abs{\ol \omega_x - \omega_x}\le(C_\e+\e\ol\eta(x))^{1/2}$; Cauchy-Schwarz then bounds the second term above absolutely by $\sqrt{\e}\int_0^\infty {(f')}^2+\frac{1}{\sqrt{\e}}\int_0^\infty{f^2(C_\e+\e\ol\eta)}$ and thus by $\sqrt{\e}\norm{f}_*^2+C_\e'\norm{f}^2$. Now combine all the terms and set $\e$ small to obtain the result.

For the boundary term $wf(0)^2$, it suffices to obtain a bound of the form $f(0)^2 \le \e\norm{f}_*^2+C''_\e\norm{f}^2$. But $f(0)^2 \le 2\norm{f'}\norm{f}$ from the proof of Fact~\ref{f.1} gives such a bound with $C''_\e = 1/\e$. 

The $L^*$ form bound follows from the fact that the $L^*$-norm dominates the $L^2$-norm. We obtain the quadratic form bound $\abs{\Hyw{f,f}}\le C\norm{f}^2_*$; it is a standard Hilbert space fact that it may be polarized to a bilinear form bound \citep[see e.g.][]{Halmos}.
\end{proof}

\begin{definition}\label{d.ee} Call $(\Lambda,f)$ an \textbf{eigenvalue-eigenfunction pair} if $f\in L^*$, $\norm{f}=1$, and for all $\ph\in C_0^\infty$ we have
\begin{equation}\label{ee}
\Hyw{\ph,f} = \Lambda\ip{\ph,f}.
\end{equation}
Note that~\eqref{ee} then automatically holds for all $\ph\in L^*$, by $L^*$-continuity of both sides.
\end{definition}

\begin{remark} This definition represents a weak or distributional version of the problem~\eqref{e_prob}. As further justification, integrate by parts to write the definition
\[
\ip{\ph',f'}-\ip{{(\ph f)}',y}+w\,\ph(0)f(0) = \Lambda\ip{\ph,f}
\]
in the form
\[
\ip{\ph',f'} - \ip{\ph',fy} +\ip{\ph',{\textstyle\int_0 f'y}}-wf(0)\ip{\ph',\1} = -\Lambda\ip{\ph',{\textstyle\int_0 f}},
\]
which is equivalent to
\begin{equation}\label{integrated}
f'(x) = wf(0) + y(x)f(x) - \int_0^x\!yf' - \Lambda\int_0^x\!f\qquad \text{a.e.\ $x$.}
\end{equation}
In the Dirichlet case the first term on the right is replaced with $f'(0)$. On the one hand \eqref{integrated} shows that $f'$ has a continuous version, and the equation may be taken to hold everywhere. In particular, $f$ satisfies the boundary condition of~\eqref{e_prob} at the origin. On the other hand, \eqref{integrated} is a straightforward integrated version of the eigenvalue equation in which the potential term has been interpreted via integration by parts. This equation will be useful in Lemma~\ref{l.one} below and is the starting point for a rigorous derivation of \eqref{SDE} in the stochastic Airy case.
\end{remark}

\begin{remark} The requirement $f\in L^*$ in Definition~\ref{d.ee} is a technical convenience. Regarding regularity, we need $f$ at least absolutely continuous to make sense of the eigenvalue equation in either an integrated or a distributional sense; we have seen, however, that solutions are in fact $C^1$. Regarding behaviour at infinity, the diffusion picture developed by \citetalias{RRV} shows a dichotomy: almost all solutions of the eigenvalue equation grow super-exponentially at infinity, except for the eigenfunctions which decay sub-exponentially.
\end{remark}

We now characterize eigenvalue-eigenfunction pairs variationally.  It is easy to see that each eigenspace is finite-dimensional: a sequence of normalized eigenfunctions must have an $L^2$-convergent subsequence by~\eqref{cbound} and Fact \ref{f.2}. By the same argument, eigenvalues can accumulate only at infinity. In fact, more is true: 

\begin{lemma}\label{l.one} For each $\Lambda\in\R$, the corresponding eigenspace is at most one-dimensional.
\end{lemma}

\begin{proof}
By linearity, it suffices to show a solution of~\eqref{integrated} with $f'(0)=f(0)=0$ must vanish identically.  Integrate by parts to write
\[
f'(x) = y(x)\!\int_0^x\!f' - \int_0^x\!yf' - \Lambda x\!\int_0^x\!f' +\Lambda\!\int_0^x\!tf'(t)dt,
\]
which implies that $\abs{f'(x)}\le C(x)\int_0^x\abs{f'}$ with some $C(x)<\infty$ increasing in $x$. Gronwall's lemma then gives $f'(x) = 0$ for all $x\ge 0$.
\end{proof}

The eigenfunction corresponding to a given eigenvalue is thus uniquely specified with the additional sign normalization $-\frac\pi2<\arg\bigl(f(0),f'(0)\bigr)\le\frac\pi2$.  We order eigenvalue-eigenfunction pairs by their eigenvalues. As usual, it follows from the symmetry of the form that distinct eigenfunctions are $L^2$-orthogonal.

\begin{proposition}\label{p.var} There is a well-defined $(k+1)$st lowest eigen\-value-eigenfunction pair $(\Lambda_k,f_k)$; it is given recursively by the minimum and minimizer in the variational problem
\[
\inf_{\substack{f\in L^*,\ \norm{f}=1,\\ f\perp f_0,\dots, f_{k-1}}} \Hyw{f,f}.
\]
\end{proposition}

\begin{remark} Since we must have $\Lambda_k\to\infty$, essentially $\{\Lambda_0,\Lambda_1,\ldots\}$ exhausts the full spectrum and the operator has compact resolvent. We do not make this precise.

\end{remark}

\begin{proof} First taking $k = 0$, the infimum $\tilde\Lambda$ is finite by~\eqref{cbound}. Let $f_n$ be a minimizing sequence; it is $L^*$-bounded, again by~\eqref{cbound}. Pass to a subsequence converging to $f\in L^*$ in all the modes of Fact~\ref{f.2}. In particular $1 = \norm{f_n}\to\norm{f}$, so $\Hyw{f,f}\ge\tilde\Lambda$ by definition. But also
\[\begin{split}
\Hyw{f,f} &= \norm{f'}^2 + \int f^2\eta + \Ip{f,\ol \omega'f} + 2\Ip{f',(\ol \omega-\omega)f} + w f(0)^2
\\&\le\, \liminf_{n\to\infty}\,\Hyw{f_n,f_n}
\end{split}\]
by a term-by-term comparison. Indeed, the inequality holds for the first term by weak convergence, and for the second term by pointwise convergence and Fatou's lemma; the remaining terms are just equal to the corresponding limits, because the second members of the inner products converge in $L^2$ by the bounds from the proof of Lemma~\ref{l.cbound} together with $L^*$-boundedness and $L^2$-convergence.  Therefore $\Hyw{f,f} = \tilde\Lambda$.

A standard argument now shows $(\tilde\Lambda,f)$ is an eigenvalue-eigenfunction pair: taking $\ph\in C_0^\infty$ and $\e$ small, put $f^\e = (f+\e\ph)/{\norm {f+\e\ph}}$; since $f$ is a minimizer, $\left.\frac{d}{d\e}\right|_{\e=0}\Hyw{f^\e,f^\e}$ must vanish;
the latter says precisely~\eqref{ee} with $\tilde\Lambda$. Finally, suppose $(\Lambda,g)$ is any eigenvalue-eigenfunction pair; then $\Hyw{g,g} = \Lambda$, and hence $\tilde\Lambda\le\Lambda$. We are thus justified in setting $\Lambda_0 = \tilde\Lambda$ and $f_0=f$.

Proceed inductively, minimizing now over $\{f\in L^*: \norm{f}=1, f\perp f_0,\dots,f_{k-1}\}$.  Again, $L^2$-convergence of a minimizing sequence guarantees that the limit remains admissible; as before, the limit is in fact a minimizer; conclude by applying the arguments of the previous paragraph in the ortho-complement. The preceding lemma guarantees that $\Lambda_0<\Lambda_1<\cdots$, and that the corresponding eigenfunctions $f_0,f_1,\ldots$ are uniquely determined.
\end{proof}

\subsection{Statement}

We are finally ready to state the main result of this section.  When we speak of an \textbf{eigenvalue-eigenvector pair $(\lambda, v)$} of an $n\times n$ matrix, we take $v\in\R^n$ embedded in $L^2(\R_+)$ as usual and normalized by $\norm{v}=1$ and $-\frac\pi2<\arg(v_0,v_1)\le\frac\pi2$. 

\begin{theorem}\label{t.conv} Suppose that $H_n$ as in~\eqref{Hn} satisfies Assumptions 1--3 and let $(\lambda_{n,k},v_{n,k})$ be its $(k+1)$st lowest eigenvalue-eigenvector pair. Define the corresponding form $\Hc_{y,w}$ as in~\eqref{Hya} and let $(\Lambda_k,f_k)$ be its a.s.\ defined $(k+1)$st lowest eigenvalue-eigenfunction pair. Then, jointly for all $k = 0,1,\ldots$ in the sense of finite-dimensional distributions, we have $\lambda_{n,k}\Rightarrow\Lambda_k$ and $v_{n,k}\Rightarrow_{L^2} f_k$ as $n\to\infty$. The convergence holds jointly over different $w_n, w$ for given $y_n,y$.
\end{theorem}

\begin{remark} Essentially, the resolvent matrices (precomposed with the corresponding finite-rank projections) are converging to the continuum resolvent in $L^2$-operator norm.  We do not define the resolvent operator here.
\end{remark}

The proof will be given over the course of the next two subsections. Recall that we proceed in the subsequential almost-sure context of the previous subsection.

\subsection{Tightness}

We will need a discrete analogue of the $L^*$-norm and a counterpart of Lemma~\ref{l.cbound} with constants uniform in $n$. For $v\in \R^n$, define the \textbf{$L^*_n$-norm} by
\[
\norm{v}_{*n}^2 \,=\,\begin{cases}\Norm{D_n v}^2 + \Norm{v\weight}^2&\text{if }w<\infty,\\
\Norm{D_n v}^2 + \Norm{v\weight}^2+ w_n v_0^2&\text{if }w=\infty,\end{cases}
\]
noting that the additional term in the Dirichlet case is nonnegative for sufficiently large $n$.

\begin{remark} As in the continuum version, the Dirichlet boundary condition must be put explicitly into the norm (see also Lemma~\ref{l.dc} below). The case considered in \citetalias{RRV} has $w_n = m_n$ in our notation; though it is somewhat hidden in the definitions, the $L^*_n$-norm used there contains a term $m_n v_0^2$.
\end{remark}	

\begin{lemma}\label{l.tight} There are constants\/ $c,C>0$ so that, for each $n$ and all\/ $v\in\R^n$,
\begin{equation}\label{tight}
c\norm{v}_{*n}^2-C\norm{v}^2 \le \ip{v,H_n v} \le C\norm{v}_{*n}^2.
\end{equation}
\end{lemma}

\begin{proof} The derivative and potential terms may be handled exactly as in \citetalias{RRV} (proof of Lemma 5.6). For the spike term $w_n v_0^2$ we recall Assumption 3.  In the $w<\infty$ case the $w_n$ are bounded, so it suffices to obtain a bound of the form $v_0^2 \le \e\norm{v}_{*n}^2+C_\e\norm{v}^2$ for each $\e>0$ where $\e,C_\e$ do not depend on $n$. Mimicking the continuum version in the proof of Fact~\ref{f.1}, we have
\[
v_0^2 = \ip{-D_n v^2,\1}
= \ip{-(D_n v)(T_n v + v),\1}
\le \ip{-(D_n v),T_n v + v}
\le 2\norm{D_n v}\norm{v},
\]
which gives the desired bound with $C_\e = 1/\e$.

In the Dirichlet case, start with \eqref{tight} but with the spike term left out (both of the form and the norm); it can be easily added back in by simply ensuring that $c\le 1$ and $C\ge 1$.
\end{proof}

\begin{remark}If $w_n\to -\infty$ then the lower bound in Lemma~\ref{l.tight} breaks down: the lowest eigenvalue of $H_n$ really is going to $-\infty$. This is the supercritical regime.
\end{remark}

\subsection{Convergence}

We begin with a lemma, a discrete-to-continuous version of Fact \ref{f.2}.

\begin{lemma}\label{l.dc}
Let $f_n\in \R^n$ with $\norm{f_n}_{*n}$ uniformly bounded. Then there exist $f\in L^*$ and a subsequence along which (i) $f_n\to f$ uniformly on compacts, (ii) $f_n\to_{L^2} f$, and (iii) $D_n f_n\to f'$ weakly in $L^2$.
\end{lemma}

\begin{proof}
Consider $g_n(x) = f_n(0) + \int_0^x D_n f_n$, a piecewise-linear version of $f_n$; they coincide at points $x = i/m_n$, $i\in\Z_+$. One easily checks that $\norm{g_n}^2_*\le 2\norm{f_n}^2_{*n}$, so some subsequence $g_n\to f\in L^*$ in all the modes of Fact~\ref{f.2}; in the Dirichlet case, the extra term in the $L^*_n$ norm guarantees that $f(0)=0$. But then also $f_n\to f$ compact-uniformly by a simple argument using the uniform continuity of $f$, $f_n\to_{L^2} f$ because $\norm{f_n-g_n}^2\le(1/3n^2)\norm{D_n f_n}^2$, and $D_n f_n\to f'$ weakly in $L^2$ because $D_n f_n = g_n'$ a.e.
\end{proof}

Next we establish a kind of weak convergence of the form $\ip{\cdot,H_n\cdot}$ to $\Hyw{\cdot,\cdot}$. Let $\Pc_n$ be orthogonal projection from $L^2$ onto $\R^n$. One can check the following: for $f\in L^2$, $\Pc_n f \to_{L^2} f$ (the Lebesgue differentiation theorem gives pointwise convergence and we have uniform $L^2$-integrability); for smooth $f$, $\Pc_n f \to f$ uniformly on compacts; further, if $f'\in L^2$ then $D_n f\to_{L^2} f'$ ($D_n f$ is a convolution of $f'$ with an approximate delta). Observe that $\Pc_n$ commutes with $R_n$ and with $D_n R_n$.

\begin{lemma}\label{l.conv} Let $f_n\to f$ be as in the hypothesis and conclusion of Lemma~\ref{l.dc}. Then for all $\ph\in C_0^\infty$ we have $\ip{\ph,H_n f_n}\to \Hyw{\ph,f}$. In particular, $\Pc_n\ph\to\ph$ in this way and so
\begin{equation}\label{conv}
\ip{\Pc_n\ph,H_n\Pc_n \ph} = \ip{\ph,H_n\Pc_n \ph}\to \Hyw{\ph,\ph}.
\end{equation}
\end{lemma}

\begin{proof} Note that if $f_n\to_{L^2} f$, $g_n$ is $L^2$-bounded and $g_n\to g$ weakly  in $L^2$, then $\ip{f_n,g_n}\to\ip{f,g}$. Therefore $\ip{\ph,D_n^\dag D_n f_n}=\ip{D_n\ph,D_n f_n}\to\ip{\ph',f'}$. The potential term converges as in \citetalias{RRV} (proof of Lemma 5.7). Moreover, the spike term converges to the boundary term:
\[
w_n f_n(0)(\Pc_n\ph)(0) \to w\,f(0)\ph(0),
\]
where in the Dirichlet case the left side vanishes for $n$ large because $\ph$ is supported away from 0.

For the second statement, the uniform $L^*_n$ bound follows from the following observations: $\Norm{(\Pc_n \ph)\weight}=\norm{\Pc_n\ph\weight}\le\norm{\ph\weight}$; for $n$ large enough that $R_n\ph=\ph$ we have $\norm{D_n\Pc_n \ph}=\norm{\Pc_n D_n\ph}\le\norm{D_n\ph}\le\norm{\ph'}$ (Young's inequality); and in the Dirichlet case, the extra term vanishes for $n$ large. The convergence is easy: $\Pc_n\ph\to\ph$ compact-uniformly and in $L^2$, and for $g\in L^2$ we have $\ip{g,D_n\Pc_n\ph}=\ip{\Pc_n g, D_n\ph}\to\ip{g,\ph'}.$
\end{proof}

Finally, we recall the argument of \citetalias{RRV} to put all the pieces together.

\begin{proof}[Proof of Theorem~\ref{t.conv}] First we show that for all $k$ we have $\underline\lambda_k=\liminf\lambda_{n,k}\ge\Lambda_k$. Assume that $\underline\lambda_k<\infty$. The eigenvalues of $H_n$ are uniformly bounded below by Lemma~\ref{l.tight}, so there is a subsequence along which $(\lambda_{n,1},\dots,\lambda_{n,k})\to(\xi_1,\dots,\xi_k = \underline\lambda_k)$. By the same lemma the corresponding eigenvector sequences have $L^*_n$-norm uniformly bounded; pass to a further subsequence so that they all converge as in Lemma~\ref{l.dc}. The limit functions are orthonormal, and by Lemma~\ref{l.conv} they are eigenfunctions with eigenvalues $\xi_k$. There are therefore $k$ distinct eigenvalues at most $\underline\lambda_k$, as required.

We proceed by induction, assuming the conclusion of the theorem up to $k-1$. First find $f_k^\e\in C_0^\infty$ with $\norm{f_k^\e - f_k}_*<\e$. Consider the vector
\[
f_{n,k} = \Pc_n f_k^\e - \sum_{j=0}^{k-1}\ip{v_{n,j},\Pc_n f_k^\e}v_{n,j}.
\]
The $L^*_n$-norm of the sum term is uniformly bounded by $C\e$: indeed, the $\norm{v_{n,j}}_{*n}$ are uniformly bounded by Lemma~\ref{l.tight}, while the coefficients satisfy $\abs{\ip{v_{n,j},f_k^\e}}\le\norm{f_k^\e-f_k}+\norm{v_{n,j}-f_j}<2\e$ for large $n$. By the variational characterization in finite dimensions, and the uniform $L^*_n$ form bound on $\ip{\cdot,H_n\cdot}$ (Lemma~\ref{l.tight}) together with the uniform bound on $\norm{\Pc_n f_k^\e}_{*n}$ (Lemma~\ref{l.conv}), we then have
\begin{equation}\label{limsup}
\limsup\lambda_{n,k}\,\le\,\limsup\frac{\ip{f_{n,k},H_n f_{n,k}}}{\ip{f_{n,k},f_{n,k}}}\,=\,\limsup\frac{\ip{\Pc_n f_k^\e,H_n\Pc_n f_k^\e}}{\ip{\Pc_n f_k^\e,\Pc_n f_k^\e}}+o_\e(1),
\end{equation}
where $o_\e(1)\to0$ as $\e\to0$. But~\eqref{conv} of Lemma~\ref{l.conv} provides $\lim \ip{\Pc_n f_k^\e,H_n\Pc_n f_k^\e} = \Hyw{f_k^\e,f_k^\e}$,
so the right hand side of~\eqref{limsup} is
\[
\frac{\Hyw{f_k^\e,f_k^\e}}{\ip{f_k^\e,f_k^\e}} + o_\e(1) \,=\, \frac{\Hyw{f_k,f_k}}{\ip{f_k,f_k}} +o_\e(1) \,=\, \Lambda_k + o_\e(1).
\]
Now letting $\e\to 0$, we conclude $\limsup\lambda_{n,k}\le\Lambda_k$.

Thus $\lambda_{n,k}\to\Lambda_k$; Lemmas~\ref{l.tight} and~\ref{l.conv} imply that any subsequence of the $v_{n,k}$ has a further subsequence converging in $L^2$ to some $g\in L^*$ with $(\Lambda_k,g)$ an eigenvalue-eigenfunction pair. But then $g = f_k$, and so $v_{n,k}\to_{L^2} f_k$.
\end{proof}

\section{Application to Wishart and Gaussian models}\label{s.3}

We now apply Theorem~\ref{t.conv} to prove Theorems~\ref{t.W} and~\ref{t.G}. The first step is to obtain the tridiagonal forms.  Then, after recalling the derivation of the scaling limit at the soft edge, we verify Assumptions 1--3 for certain scalings of the perturbation.

\subsection{Tridiagonalization}

We explain how to tridiagonalize a rank one spiked real Wishart matrix; the algorithm is basically the usual one described by \citet{Trotter} with a few careful choices. We restrict for the moment to the case $n\ge p$, but lift this restriction in the Remark~\ref{r.duality} below.  For a given $p\times n$ data matrix $X$ we will construct a pair of orthogonal matrices $O\in O(p),\ O'\in O(n)$ so that $W = OXO'$ becomes lower bidiagonal; then $X$ and $W$ have the same singular values and $WW^\dag$ is a symmetric tridiagonal matrix with the same eigenvalues as $XX^\dag$.  Further, the structure of $X$ and $O,O'$ will be such that the entries of $W$ are independent with explicit known distributions.

We build up $O$ and $O'$ as follows.  Let $e_1,\dots, e_p\in\R^p$ be the standard basis of column vectors and $\tilde e_1,\dots,\tilde e_n\in\R^n$ the standard basis of row vectors.
\begin{itemize}\itemsep 0pt
\item First, reflect (or rotate) the top row of $X$ into the positive $\tilde e_1$ direction via right multiplication by $O_1'\in O(n)$, chosen independently of the other rows. This row becomes $\sqrt{\ell}\,\wt\chi_n\tilde e_1$, where $\wt\chi_n$ is a $\mathrm{Chi}(n)$ random variable (i.e.\ distributed as the length of an $n$-dimensional standard normal vector); the other rows remain independent standard normal vectors, since their distribution is invariant under an independent reflection.

\item Next, reflect the first column of $XO_1'$ as follows: leaving $\langle e_1\rangle$ invariant, reflect the orthogonal $\langle e_2,\dots,e_p\rangle$ component of the column into the positive $e_2$ direction via left multiplication by $O_1\in\{I_1\}\oplus O(p-1)$, chosen independently of the other columns. This component of the column becomes $\chi_{p-1}e_2$ where $\chi_{p-1}\sim \mathrm{Chi}(p-1)$, independent of $\wt\chi_n$. The same components of the other columns remain independent standard normal vectors, while the first row is untouched.

\item Reflect the second row of $O_1XO_1'$ as follows: leaving $\langle \tilde e_1\rangle$ invariant, reflect the orthogonal component of the row into the positive $\tilde e_2$ direction via right multiplication by $O_2'\in\{I_1\}\oplus O(n-1)$, chosen independently of the other rows.

\item Reflect the second column of $O_1XO_1'O_2'$ as follows: leaving $\langle e_1, e_2\rangle$ invariant, reflect the orthogonal component of the column into the positive $e_3$ direction via left multiplication by $O_2\in\{I_2\}\oplus O(p-2)$, chosen independently of the other columns.

\item Continue in this way, alternately reflecting rows and columns while leaving the results of previous steps untouched.
\end{itemize}
The result is that with $O' = O_1'\cdots O_p'$ and $O = O_{p-1}\cdots O_1$ we have
\begin{equation*}
W = OXO' = \begin{bmatrix}
\sqrt{\ell}\wt\chi_{n}
\\ \chi_{p-1} & \wt\chi_{n-1}
\\ & \ddots & \ddots
\\ & & \chi_2 & \wt\chi_{n-p+2}
\\ & & & \chi_1 & \wt\chi_{n-p+1}
\end{bmatrix},
\end{equation*}
where $\{\wt\chi_{n-j}\}_{j=0}^{p-1}$ and $\{\chi_{p-j}\}_{j=1}^{p-1}$ are independent Chi random variables of parameters given by their indices. We have truncated the $n - p$ rightmost columns of zeros to obtain a $p\times p$ matrix, leaving the product $WW^\dag$ unchanged. We will actually work with $W^\dag W$ below, which has the same eigenvalues.

\begin{remark}\label{r.duality}
Attempting the above procedure in the case $n<p$ produces a lower bidiagonal matrix $W$ with $n+1$ nonzero rows.  The matrix $W^\dag W$ is now $n\times n$, has the same \emph{nonzero} eigenvalues as $XX^\dag$, and looks just like it does in the $n\ge p$ case except for a discrepancy in the bottom-right corner. The two cases may in fact be unified if one agrees that $\chi_0 = 0$; then $W$ is $(n\wedge p+1)\times(n\wedge p)$ and has the form~\eqref{Wb} with $\beta = 1$, while $W^\dag W$ is $(n\wedge p)\times(n\wedge p)$.
\end{remark}

The same algorithm will tridiagonalize a rank one spiked complex (resp.\ quaternionic) Wishart matrix by unitary (resp.\ symplectic or hyperunitary) conjugations. The lower bidiagonal matrix will be $W^{\beta,\ell}_{n, p}$ from~\eqref{Wb} with $\beta = 2$ (resp.\ 4).

The perturbed GOE/GUE/GSE ensembles are even easier to tridiagonalize; as in the Wishart case, the usual procedure of \citet{Trotter} works without modification. Starting with an $n\times n$ GOE matrix $M$ with a perturbation in the (1,1) entry, the upshot is that for certain $O_1,\dots, O_{n-1}$ with $O_j\in\{I_j\}\oplus O(n-j)$ the conjugated matrix $O_{n-1}\cdots O_1 M O_1^\dag\cdots O_{n-1}^\dag$ has the form~\eqref{Gb} with $\beta=1$. We do not detail it further here.

\subsection{Scaling limit}

Consider the $\ell$-spiked $\beta$-Laguerre ensemble $S = W^\dag W$ with $W = W_{n,p}=W^{\beta,\ell_{n,p}}_{n, p}$ as in~\eqref{Wb}, recalling that $S_{n,p}$ is $(n\wedge p)\times(n\wedge p)$. The diagonal and off-diagonal processes of $\beta S$ are
\begin{gather*}\ell_{n,p}\wt\chi_{\beta n}^2+\chi_{\beta(p-1)}^2,\quad\wt\chi_{\beta(n-1)}^2+\chi_{\beta(p-2)}^2,\quad\wt\chi_{\beta(n-2)}^2+\chi_{\beta(p-3)}^2,\quad\ldots
\\ \wt\chi_{\beta(n-1)}\chi_{\beta(p-1)},\quad\wt\chi_{\beta(n-2)}\chi_{\beta(p-2)},\quad\ldots
\end{gather*}
respectively. The usual centering and rescaling for fluctuations at the soft edge---as well as the operator limit itself---can be predicted using the approximations
\begin{equation*}
\chi_k\,\approx\,\sqrt{k}+\sqrt{1/2}\,g, \qquad\chi_k^2\,\approx\,k + \sqrt{2k}\,g,
\end{equation*}
valid for $k$ large, where $g$ is a suitably coupled standard Gaussian. We briefly reproduce the heuristic argument.

To leading order, the top-left corner of $S$ has $n+p$ on the diagonal and $\sqrt{np}$ on the off-diagonal. So the top-left corner of
\begin{equation*}
\frac{1}{\sqrt{np}}\Bigl(S - \bigl(\sqrt{n} + \sqrt{p}\bigr)^2 I\Bigr)
\end{equation*}
is approximately an unscaled discrete Laplacian. If time is scaled by $m^{-1}$, space has to be scaled by $m^2$ for this to converge to $\frac{d^2}{dx^2}$.  The next order terms for the $j$'th diagonal and off-diagonal entries of $S$, where $j\ll n\wedge p$, are respectively
\begin{gather*}
\tfrac{1}{\sqrt\beta}\bigl(\sqrt{2 n}\,\wt g_{n-j+1} + \sqrt{2p}\,g_{p-j} - 2j\bigr),
\\ \tfrac{1}{\sqrt\beta}\bigl(\sqrt{p/2}\,\wt g_{n-j}+\sqrt{n/2}\,g_{p-j} -1/2(\sqrt{p/n} +\sqrt{n/p})j\bigr).
\end{gather*}
(we have indexed the $g$'s to match the corresponding $\chi$'s).
The total noise per unit (unscaled) time is like $\trb\bigl(\sqrt{n} + \sqrt{p}\bigr)g$; convergence to $\trb$ times standard Gaussian white noise $b_x'$ then requires $\bigl(\sqrt{n}+\sqrt p\bigr)m_n^2/\sqrt{np} = m^{1/2}$. The averaged part of the potential requires $\bigl(2 + \sqrt{p/n}+\sqrt{n/p}\bigr)m^2/\sqrt{np} = m^{-1}$ to converge to the function $-x$.  Fortunately these two scaling requirements match perfectly; we set
\begin{equation*}
m_{n,p} = \left(\frac{\sqrt{np}}{\sqrt{n}+\sqrt{p}}\right)^{2/3},\quad H_{n,p} = \frac{m_{n,p}^2}{\sqrt{np}}\left(\left(\sqrt{n}+\sqrt{p}\right)^2 I_{n\wedge p} - S_{n,p}\right)
\end{equation*}
and set the integrated limiting potential to
\begin{equation*}
y(x) = \half x^2 + \trb b_x
\end{equation*}
where $b_x$ is a standard Brownian motion. Note that
\[
2^{-2/3}(n\wedge p)^{1/3}\,\le\, m\,\le\, (n\wedge p)^{1/3},
\]
so the conditions $m\to\infty$, $m=o(n\wedge p)$ are met by merely having $n,p\to\infty$ together.

We now carefully decompose $H_{n,p}$ as in~\eqref{Hn}.  In \eqref{Hn1},\eqref{Hn2} there is a little freedom between $y_{n,1;1}$ and $w_n$, but only in to an additive constant in $y_{n,1}$ that tends to zero in probability anyway. Thus we may as well set $y_{n,1;1}=0$ to fix $w_{n}$ and $y_{n,i}$. Assumptions~1 and~2 (the CLT~\eqref{a1} and required tightness~\eqref{a21}--\eqref{a23} for the potential terms $y_{n,i}$) are then verified exactly as in the final subsection of \citetalias{RRV}\@.

It remains to consider Assumption~3. We have
\[
w_n = m_{n,p}\left(1+\sqrt{\frac n p}\biggl(1-\ell_{n,p}\frac{\wt\chi^2_{\beta n}}{\beta n}\biggr)+\sqrt{\frac p n}\biggl(1-\frac{\chi^2_{\beta(p-1)}}{\beta p}\biggr)\right).
\]
First order heuristics suggest we take $\ell_{n,p}$ to satisfy
\[
\ol w_n = m_{n,p}\left(1+\sqrt{\frac n p}\left(1-\ell_{n,p}\right)\right)\,\to\,w\in(-\infty,\infty]\qquad\text{as }n\wedge p\to\infty
\]
as in \eqref{Wspike}. We want to show that, in this case, $w_n\to w$ in probability; it is certainly enough to show that $w_n-\ol w_n\to 0$ in probability.

Second order heuristics say the error terms are on the order $(n\wedge p)^{-1/6}$ or $m^{-1/2}$, and $L^2$ estimates easily provide the rigour. All we need is that $\chi_k^2$ has mean $k$ and variance $2k$.
We have 
\[
w_n-\ol w_n = -\frac{m\ell}{\beta\sqrt{np}}\left(\chi^2_{\beta n}-\beta n\right) + \frac{m}{\beta\sqrt{np}}\left(\beta(p-1) - \chi_{\beta(p-1)}^2\right) + \frac{m}{\sqrt{np}}.
\]
Using that $\ell\le 1+2\sqrt{p/n}$, the mean square of the first term is $O(m^2/p + m^2/n)$, which is $O(m^{-1})$.   The mean square of the second term is $O(m^2/n)$, again $O(m^{-1})$.
The last term is negligible.  This completes the proof of Theorem~\ref{t.W}.

Turning now to the perturbed $\beta$-Hermite ensemble, take $G_n=G_n^{\beta,\mu_n}$ as in~\eqref{Gb}. With heuristic motivation similar to that in the previous proof, set
\[
m_n = n^{1/3} ,\qquad H_n = \frac{m_n^2}{\sqrt{n}}\left(2\sqrt{n}I_n - G_n\right)
\]
and $y(x)$ as before.  Decompose $H_n$ as in~\eqref{Hn}. Again, the verification of Assumptions 1 and 2 on $y_{n,i}$ proceeds as in \citetalias{RRV} (Lemmas 6.2, 6.3). Moving on to Assumption 3, we have
\[
w_n = m_n\left(1 - \bigl(\mu_n + \sqrt{2/\beta n}\,g_1\bigr)\right).
\]
Putting
\[
\ol w_n = m_n\left(1-\mu_n\right)
\]
as in~\eqref{Gspike}, the difference is $w_n-\ol w_n = -n^{-1/6}\sqrt{2/\beta}\,g_1$. It follows that $w_n-\ol w_n\to 0$ in probability, which completes the proof of Theorem~\ref{t.G}.

\section{Alternative characterizations of the laws}\label{s.4}

In this section we prove Theorem~\ref{t.char} and its extension to higher eigenvalues.

\subsection{Diffusion} 

The diffusion characterization is developed in \citetalias{RRV}\@.  The starting point is an application of the classical Riccati map $p = f'/f$ to the eigenvalue equation~\eqref{e_prob}, or rigorously to~\eqref{integrated}; the result is the first order differential equation
\begin{equation}\label{Ric}
p'(x) = x-\lambda+\trb b'(x)-p^2(x)
\end{equation}
understood also in the integrated sense. The boundary condition at the origin becomes the initial value
\[
p(0)=w,
\]
and a zero of $f$ would have $p$ explode to $-\infty$ and immediately restart at $+\infty$.

One can in fact construct the solution for any $\lambda\in\R$.  One way to see this is to introduce the variable $q(x) = p(x)+\trb b(x)$; the ODE
\begin{equation}\label{ODE}
q' = x-\lambda-\bigl(q+\trb b\bigr)^2
\end{equation}
is classical and the Picard existence and uniqueness theorem applies.  Although solutions can explode to $-\infty$ in finite time, this is not a problem if we consider the values on the projective line. Behaviour through $\infty$ can then be understood in the other coordinate $\tilde q = 1/q$, which evolves as
\[
\tilde q' = \bigl(1+\trb b\tilde q\bigr)^2 - (x-\lambda)\tilde q^2;
\]
in particular, $\tilde q'=1$ whenever $\tilde q=0$. The solution can thus be continued for all time. Moreover, it depends monotonically and continuously on the the parameter $\lambda$, uniformly on compact time-intervals with respect to the topology of the projective line. Following classical Sturm oscillation theory one can argue that almost surely, for all $\lambda\in\R$, \emph{the number of eigenvalues strictly less than $\lambda$ equals the number of explosions of $p$ on $\R_+$}.

For a fixed $\lambda$, the Riccati equation~\eqref{Ric} may also be understood in the It\=o sense; by translation equivariance the time-shift $x\mapsto x-\lambda$ produces the same path measure as the It\=o diffusion \eqref{SDE} started at time $x_0 = -\lambda$.  Writing $\kappa_{(x_0,w_0)}$ for the distribution of the first explosion time of $p_x$ under $\P_{(x_0,w_0)}$---an improper distribution with some mass on $\infty$---we have $\P_{\beta,w}(\Lambda_0<\lambda) = \kappa_{(-\lambda,w)}(\R)$ or $F_{\beta,w}(x) = \kappa_{(x,w)}(\{\infty\})$ as in \eqref{diffusion}.  More generally, the strong Markov property gives
\begin{equation}\label{diffusion_h}
\P_{\beta,w}(-\Lambda_{k-1}>x) = \int_{\R^k}\!\kappa_{(x,w)}(dx_1)\,\kappa_{(x_1,\infty)}(dx_2)\cdots\,\kappa_{(x_{k-1},\infty)}(dx_k).
\end{equation}

The stated path properties of~\eqref{SDE} appear also in \citetalias{RRV} (Propositions 3.7 and 3.9).

\subsection{Boundary value problem}

Briefly, the boundary value problem is just the Kolmogorov backward equation for a hitting probability of the diffusion.  We assume the diffusion representation $F_{\beta,w}(x)=\kappa_{(x,w)}(\{\infty\})$ for the distribution of $-\Lambda_0$.

\begin{lemma}\label{c.w} For each fixed $x$, $F_{\beta,w}(x)$ is nondecreasing and continuous in $w\in(\infty,\infty]$ and tends to zero as $w\to-\infty$.
\end{lemma}

\begin{remark} There are in fact almost-sure counterparts of these assertions that describe how $\Lambda_0$ depends on $w$ for each Brownian path, but we do not need them here.
\end{remark}

\begin{proof} The monotonicity is a consequence of uniqueness of the diffusion path from each space-time point: two paths started from $(x,w_0)$ and $(x,w_1)$ with $w_0<w_1$ never cross, so if the upper path explodes to $-\infty$ then the lower path must do so as well. The continuity is a general property of statistics of diffusions: $\kappa_{(x,p_x)}(\{\infty\})$ is a martingale, so $F_{\beta,w}(x)$ is in fact space-time harmonic. (Again, the behaviour at $w=+\infty$ may be understood by changing coordinates.) 

The final assertion is that for fixed $x_0$ explosion becomes certain as $w\to-\infty$. It may be verified by a domination argument involving the ODE~\eqref{ODE} (time-shifted as above so that $\lambda = 0$ and the initial time is $x_0$), whose paths explode simultaneously with those of \eqref{SDE}.  Given $\e>0$, let $M$ be such that $\P(\sup_{x\in[x_0,x_0+1]}\abs{b_x}>M)<\e$. It is easy to check that for $r_0$ sufficiently negative, the solution of $r' = x -(r+M)^2$ with initial value $r(x_0) = r_0$ explodes to $-\infty$ before time $x_0+1$. Now consider the solution of $q' = x - (q+b)^2$ with $q(x_0) \le r_0\le -M$. With probability $1-\e$ we have $q'(x)\le r'(x)$ whenever $q(x) = r(x)$, so the paths never cross and $q$ explodes as well.
\end{proof}

\begin{proof}[Proof of Theorem \ref{t.char} (ii)] Writing $L$ for the space-time generator of the SDE \eqref{SDE}, the PDE \eqref{PDE} is simply the equation $LF = 0$. Therefore the hitting probability $F(x,w) = F_{\beta,w}(x)$ satisfies the PDE. The boundary behaviour \eqref{BC} follows from Lemma~\ref{c.w} and the fact that $F(\cdot,w)$ is a distribution function for each $w$.  Specifically, the lower part of the boundary behaviour follows from the fact that $F(x,w)$ is increasing in $x$ and $F(x,w)\to 0$ as $w\to-\infty$ for each $x$. The upper part follows from the fact that $F(x,w)$ is increasing in $w$ and $F(x,w)\to 1$ for fixed $w$ as $x\to\infty$. 

Toward uniqueness, suppose $\tilde F(x,w)$ is another bounded solution of \eqref{PDE},\eqref{BC}. By the PDE, $\tilde F(x,p_x)$ is a local martingale under $\P_{(x_0,w_0)}$ and thus a bounded martingale.  Let $T$ be the lifetime of the diffusion; optional stopping gives $\tilde F(x, w) = \E_{(x,w)} \tilde F(T\wedge t,p_{T\wedge t})$ for all $t\ge x$. Taking $t\to\infty$, we conclude by bounded convergence, the boundary behaviour of $\tilde F$ and the stated path properties of the diffusion that $\tilde F(x,w)$ is the non-explosion probability. That is, $\tilde F = F$.
\end{proof}

As promised, we indicate how the laws of the higher eigenvalues $\Lambda_1,\Lambda_2,\ldots$ may be characterized in terms of the PDE \eqref{PDE}.  The characterization is inductive and follows from~\eqref{diffusion_h} by reasoning just as in the preceding proof.

\begin{theorem} Let $F_{(0)}(x,w) = \P_{\beta,w}(-\Lambda_{0}<x)$. For each $k=1,2,\ldots$, the boundary value problem
\begin{gather*}
\frac{\del F}{\del x} + \frac{2}{\beta}\frac{\del^2 F}{\del w^2} + \bigl(x-w^2\bigr)\frac{\del F}{\del w} = 0\qquad\text{ for }(x,w)\in\R^2,
\\
F(x,w)\to\begin{cases} 1 &\text{ as }x,w\to\infty\text{ together},
\\F_{(k-1)}(x_0,+\infty)&\text{ as }w\to-\infty\text{ while }x\to x_0\in\R
\end{cases}
\end{gather*}
has a unique bounded solution $F_{(k)}$, and we have $\P_{\beta,w}(-\Lambda_{k}<x) = F_{(k)}(x,w)$ for $w\in(-\infty,\infty)$; further, $\P_{\beta,\infty}(-\Lambda_{k}<x) = \lim_{w\to\infty} F_{(k)}(x,w)$.
\end{theorem}

\section{Connection with Painlev\'e II}\label{s.5}

We now prove Theorem~\ref{t.id} and Corollary~\ref{c.id}. We will need some standard facts about the function $u(x)$ defined by \eqref{PII},\eqref{HM} and the derived functions $v(x),\ E(x),\ F(x)$ defined in \eqref{v},\eqref{EF}.

\begin{fact}\label{f.p1} The following hold:
\begin{enumerate}[(i)]\itemsep 1pt
\item $u > 0$ on $\R$ and $u'/u\sim-\sqrt{x}$ as $x\to+\infty$.
\item $E$ and $F$ are distribution functions.
\item $E(x) = O(e^{-cx^{3/2}})$ for some $c>0$ as $x\to+\infty$.
\end{enumerate}
\end{fact}

We will also take for granted some additional information about the functions $f(x,w)$, $g(x,w)$ defined by \eqref{w_lax},\eqref{IC}.

\begin{fact}\label{f.p2} The following hold.
{\abovedisplayskip 0pt
\begin{enumerate}[(i)]\itemsep 1pt
\item For each $x\in\R$,
\begin{align}\label{w_up}
\lim_{w\to+\infty}\begin{pmatrix}f\\g\end{pmatrix} &= \begin{pmatrix}1\\0\end{pmatrix},
\\\label{w_down}
\lim_{w\to-\infty}\begin{pmatrix}f\\g\end{pmatrix} &= \begin{pmatrix}0\\0\end{pmatrix}.
\end{align}

\item For each $w\in\R$,
\begin{equation}
\label{x_lax}\frac{\del}{\del x}\begin{pmatrix}f\\g\end{pmatrix} =\begin{pmatrix}0&u(x)\\u(x)&-w\end{pmatrix}\begin{pmatrix}f\\g\end{pmatrix}.
\end{equation}

\item There is the identity
\begin{equation}\label{duality}
g(x,w) = f(x,-w)e^{\frac13 w^3-xw}.
\end{equation}

\item For fixed $w\in\R$,
\begin{align}\label{x_right}
f(x,w)  \to 1 \quad&\text{as }x\to +\infty;
\\\label{x_left}
f(x,w) > 0 \quad&\text{for $x$ sufficiently negative}.
\end{align}
\end{enumerate}}
\end{fact}

These properties follow from an analysis of the associated Riemann-Hilbert problem with the special monodromy data corresponding to the Hastings-McLeod solution \citep[see][]{FIKN}. They are proved in \citet{BR2} except for (iv) which goes back to \cite{DZ}. Interestingly~\eqref{IC} and~\eqref{w_up} are interchangeable in that the latter also uniquely determines a solution of~\eqref{w_lax}; this fact does not depend on the specific solution of~\eqref{PII} specified by~\eqref{HM}. By contrast, \eqref{w_down} does depend on~\eqref{HM}. Equations \eqref{w_lax},\eqref{x_lax} constitute a so-called \emph{Lax pair} for the Painlev\'e II equation~\eqref{PII}. (It is in fact a simple transformation of the standard Flaschka-Newell Lax pair.)  The consistency condition of this overdetermined system---i.e.\ that the partials commute---is the Painlev\'e II equation.

\begin{proof}[Proof of Theorem \ref{t.id}, $\beta=2$ case.] 
Let $\tilde F_2(x,w)$ denote the right-hand side of~\eqref{id2}. Using~\eqref{EF}, \eqref{w_lax} and \eqref{x_lax}, we check that that $\tilde F_2$ solves the PDE~\eqref{PDE} with $\beta=2$: compute
\begin{align*}\frac{\del\tilde F_2}{\del x} &= \Bigl\{vf + ug\Bigr\}F
\\ \frac{\del\tilde F_2}{\del w} &= \Bigl\{u^2 f + \bigl(-wu-u'\bigr)g\Bigr\}F
\\ \frac{\del^2\tilde F_2}{\del w^2} &= \Bigl\{\bigl(u^4+w^2u^2 - (u')^2\bigr)f + \bigl(-u + (wu+u')(x-w^2)\bigr)g\Bigr\}F
\end{align*}
and substitute. The coefficient of $g$ vanishes and the coefficient of $f$ is
\[
v + u^4 - (u')^2 + xu^2.
\]
Differentiating, we see that this quantity is constant by~\eqref{PII}. As all terms vanish in the limit as $x\to\infty$, the constant is zero.

We must check that $\tilde F_2$ is bounded and that it has the boundary behaviour~\eqref{BC}.  To this end we claim $f,g>0$ on $\R^2$. Fixing $w$, \eqref{x_left},\eqref{duality} cover $x$ sufficiently negative. Now \eqref{x_lax} shows $f$ increases at least until $x_0 = \min\{x:g(x,w)=0\}$. But if $x_0$ exists then \eqref{x_lax} shows $\tfrac{\del g}{\del x}(x_0) > 0$, a contradiction. This proves the claim. It now follows from \eqref{x_lax} that $\tfrac{\del f}{\del x} > 0$. From \eqref{x_right} we deduce that $f\le 1$; in particular $f$ is bounded, and hence so is $\tilde F_2$.  Furthermore, for a given $x\in\R$ and $\e>0$, \eqref{w_up} yields $w_+$ such that $f>1-\e$ on $[x,\infty)\times[w_+,\infty)$, and \eqref{w_down} yields $w_-$ such that $f<\e$ on $(-\infty,x]\times(-\infty,w_-]$. Using that $F(x)$ is a distribution function, \eqref{BC} follows.
\end{proof}

\begin{proof}[Proof of Theorem \ref{t.id}, $\beta=4$ case.]

That the right-hand side $\tilde F_4$ of \eqref{id4} satisfies the PDE \eqref{PDE} with $\beta=4$ may be verified just as in the $\beta=2$ case; the computation is more tedious but the result is very similar and the final step is the same.

It is a little more work to get boundedness and the boundary behaviour~\eqref{BC} this time. Dropping the scale factors on $x,w$, consider
\[
G = F^{-1/2}\tilde F_4 = \half\bigl(E^{-1/2} + E^{1/2}\bigr)f + \half\bigl(E^{-1/2}-E^{1/2}\bigr)g.
\]
Clearly $G>0$. For fixed $w$, $G\to 1$ as $x\to\infty$ by \eqref{x_right} and the fact that $E^{-1/2}-E^{1/2} = O(e^{-cx^{3/2}})$ while $g = O(e^{wx})$ from \eqref{duality}. Now by \eqref{x_lax} we have
\[
\frac{\del G}{\del x} = \half\bigl(E^{-1/2}+E^{1/2}\bigr)\bigl(\half ug\bigr) + \half\bigl(E^{-1/2}-E^{1/2}\bigr)\bigl(\half uf-wg\bigr),
\]
which is positive for $w\le 0$. Boundedness in the lower half-plane $\{w\le0\}$ follows, as does the lower boundary behaviour using \eqref{w_down}.

From \eqref{duality} we immediately see $g\le 1$ on $\{x\ge 0,\,0\le w\le\sqrt{3x}\}$. By Lemma~\ref{c.w}, $\frac{\del}{\del w}F_{\beta,w}(x)\ge 0$. The $\beta=2$ case of the present theorem then implies that $\frac{\del f}{\del w}\ge 0$. From \eqref{w_lax} we conclude $g\le u/(w+u'/u)$ provided the denominator is positive.  But $u'/u\sim-\sqrt{x}$ as $x\to+\infty$, so there is $x_1$ such that $u'/u \ge -\sqrt{2x}$ for $x\ge x_1$. The latter bound for $g$ therefore implies that $g$ is bounded on $\{x\ge x_1,\,w\ge\sqrt{3x}\}$. Moreover, for any $x_0<x_1$ we have that $u$ and $u'/u$ are bounded on the interval $x_0\le x\le x_1$, so $g$ is bounded uniformly over these $x$ for all $w$ sufficiently large.  Putting these bounds together we conclude $g$ is bounded on all right half-planes $\{x\ge x_0\}$, and the same then follows for $\tilde F_4$.

The upper boundary behaviour follows as well.  Indeed, as $x,w\to\infty$ together the coefficient of $g$ vanishes while the coefficient of $f$ tends to 1; the $g$-term then vanishes while the $f$-term tends to 1 as in the $\beta=2$ case.

It remains to show $\tilde F_4$ is bounded on the whole plane; it suffices to bound $\tilde F_4$ on the upper-left quadrant $Q = \{x\le 0,\,w\ge 0\}$. Here we can use the fact that $\tilde F_4$ solves the PDE. With notation as in Theorem \ref{t.char} we have that $\tilde F_4(x,p_x)$ is a local martingale under $\P_{(x_0,w_0)}$. By boundedness on right half-planes, it is in fact a bounded martingale. Using that paths explode only to $-\infty$, optional stopping gives the representation $\tilde F_4(x_0,w_0) = \E_{(x_0,w_0)} \tilde F_4(T,p_T)$ where $T=\inf\{x:(x,p_x)\notin Q\}$. The bound thus extends to $Q$.
\end{proof}

\begin{proof}[Proof of Corollary \ref{c.id}] These identities are straightforward consequences of the theorem, \eqref{IC} and \eqref{w_up}.
\end{proof}

\bigskip
\noindent\textbf{Acknowledgements}
\phantomsection\addcontentsline{toc}{section}{Acknowledgements}\quad The second author is very grateful to Jos\'e Ram\'irez for conversations that helped this project go forward.  The first author is indebted to Alexander Its for his patient and thorough explanations.  We would like to thank Jinho Baik, Alexei Borodin, Peter Forrester, Arno Kuijlaars, Eric Rains, Brian Rider, Brian Sutton, Dong Wang and Ofer Zeitouni for interesting and helpful discussions, as well as AIM and MSRI for providing stimulating environments in December 2009 and September 2010 workshops. The work of the first author was supported in part by an NSERC postgraduate scholarship held at the University of Toronto, and that of the second author by the Canada Research Chair program and the NSERC DAS program.

\setstretch{1}
\phantomsection\addcontentsline{toc}{section}{References}
\bibliographystyle{dcu}
\bibliography{./spiked}

\bigskip\bigskip\bigskip\noindent
\begin{minipage}{0.49\linewidth}
Alex Bloemendal
\\Department of Mathematics
\\Harvard University
\\Cambridge, MA 02138
\\{\tt alexb@math.harvard.edu}
\end{minipage}
\begin{minipage}{0.49\linewidth}
B\'alint Vir\'ag
\\Departments of Mathematics and Statistics
\\University of Toronto
\\Toronto ON~~M5S 2E4, Canada
\\{\tt balint@math.toronto.edu}
\end{minipage}

\end{document}